\documentclass[11pt]{amsart}
\usepackage{times}
\usepackage[T1]{fontenc}
\usepackage{amssymb, amsthm, amsmath}
\usepackage{mathrsfs}
\usepackage{mathtools}
\usepackage{comment}
\usepackage[active]{srcltx}

\usepackage{todonotes}

\usepackage{setspace}
\usepackage{geometry}

\usepackage{hyperref}

\DeclareMathOperator{\acl}{acl}
\DeclareMathOperator{\dcl}{dcl}

 \DeclareMathOperator{\dom}{dom}

\DeclareMathOperator{\eq}{eq}

\DeclareMathOperator{\tp}{tp}

\DeclareMathOperator{\cl}{cl}

\newtheorem{introtheorem}{Theorem}

\newtheorem{theorem}{Theorem}[section]

\newtheorem{claim}{Claim}[theorem]

\newtheorem{corollary}[theorem]{Corollary}

\newtheorem{fact}[theorem]{Fact}
\newtheorem{lemma}[theorem]{Lemma}

\newtheorem{proposition}[theorem]{Proposition}

\newtheorem*{gen-dif}{\fbox{{\large A}} \hypertarget{Agen-dif}{Gen-Dif}}

\newtheorem*{min-balln}{\fbox{{\large A}} \hypertarget{Amin-ball}{Cballs}}

\theoremstyle{definition}
\newtheorem{definition}[theorem]{Definition}
\newtheorem{example}[theorem]{Example}
\newtheorem{remark}[theorem]{Remark}
\newtheorem{question}[theorem]{Question}
\newtheorem{notation}[theorem]{Notation}

\newcommand{\Nn}{{\mathbb{N}}}

\newcommand{\CD}{{\mathcal D}}

\newcommand{\CK}{{\mathcal K}}
\newcommand{\CN}{{\mathcal N}}

\newcommand{\CM}{{\mathcal M}}

\newcommand{\CC}{{\mathcal C}}

\newcommand{\CO}{{\mathcal O}}

\newcommand{\0}{\emptyset}

\newcommand{\rest}{\upharpoonright}

\renewcommand{\phi}{\varphi}

\newcommand{\widebar}[1]{\mkern 4mu\overline{\mkern-4mu#1\mkern-2mu}\mkern 2mu}

\def\sub{\subseteq}

\def\conv{\overline{\mathrm{conv}}}
\def\Im{\mathrm{Im}}

\newenvironment{claimproof}[1][\proofname]
{%
	\proof[#1]%
	\renewcommand*\qedsymbol{$\square$ (claim)}%
}
{%
	\endproof%
}

\title{On differentiability in Weakly o-minimal expansions of fields}

\author{Juan Pablo Acosta}
\address{Department of Mathematics, University of Pamplona, Pamplona, Colombia}
\email{juan.acostal@unipamplona.edu.co}

\author{Yatir Halevi}
\address{Faculty of Mathematics, Technion - Israel Institute of Technology, Haifa, Israel}
 \email{yatirh@math.technion.ac.il}

\author{Assaf Hasson}
\address{Department of Mathematics, Ben Gurion University of the Negev, Be'er-Sheva, Israel}
\email{hassonas@math.bgu.ac.il}

\author{Will Johnson}
\address{School of Mathematical Sciences, Fudan University, Shanghai, China}
\email{willjohnson@fudan.edu.cn}

\author{Ya'acov Peterzil}
\address{Faculty of Natural Sciences, Department of Mathematics, University of Haifa, Haifa, Israel}
\email{kobi@math.haifa.ac.il}

\date{\today}
\begin{document}

\thanks{The second and third authors were partially supported by ISF grant No. 555/21. The second author was further supported by The Israel Science Foundation (grant No. 295/26). The fourth author was supported by the National Natural Science Foundation of China (grant No. W2532009).  The fifth author was partially supported by ISF grant 290/19.}

	\maketitle
\begin{abstract}
    We show that every definable function in a weakly o-minimal expansion of a field $K$ is generically differentiable, where the derivative takes values in the Cauchy completion of $K$,  i.e. it is a definable nonvaluational cut. As a consequence, every definable field is definably isomorphic to either  $K$ or its algebraic closure.
\end{abstract}
    
\section{Introduction}

O-minimal structures are linearly ordered structures in which every unary definable subset is a finite union of points and intervals. Weakly o-minimal structures were introduced by Dickmann \cite{dickman}, as a generalization of o-minimality which allows for valuation theory: They are linearly ordered structures in which every unary definable subset is a finite union of convex sets. A theory of a linearly ordered structure is called weakly o-minimal if every elementary equivalent structure is such. Macpherson, Marker and Steinhorn initiated a systematic study of weakly o-minimal structures and theories in \cite{MacMaSt}, proving that weakly o-minimal ordered fields are real closed, and listing several problems, most of which remained open until recently.

In a recent preprint, \cite{Joh:exchange}, the fourth author answers two of the main questions of that paper, in the setting of weakly o-minimal theories expanding fields. He proves that model theoretic $\acl(\cdot)$
satisfies the exchange principle. He also constructs a weakly o-minimal theory of fields equipped with a definable function that is nowhere differentiable, showing that weakly o-minimal fields may fail generic differentiability. 

 
This example raises, therefore, the question of whether definable functions in weakly o-minimal expansions of fields exhibit any general tameness properties beyond generic continuity. The main result of the current paper implies that this example is the worst that could happen, in the following sense: Definable functions are generically differentiable, if one is willing to let the derivative take values in a larger field, as explained below.

 
 Recall that a cut $ (C_1,C_2)$ in a real closed 
 field  $K $ is called {\em valuational} if $C_1$ is additively invariant under some nonzero elements of $K$; it is called {\em nonvaluational} otherwise. 
The collection of all nonvaluational cuts in $K$, which we denote by $\widehat K$, extends $K$ and  can be identified with the Cauchy completion of $K$, thus it is a real closed field itself. When we restrict to the {\bf definable} nonvaluational cuts, we obtain a  real closed subfield  $\widebar K_{nv}\sub \widehat K$ containing $K$; it is not a definable object but rather a collection of definable sorts. See Section \ref{s:PreNot}.

 In the example from \cite{Joh:exchange}, for every $x\in K$, the expected derivative $\lim_{h\to 0}\frac{f(x+h)-f(x)}{h}$ does not belong to $K$, but instead determines a nonvaluational cut in $K$ and thus belongs to $\widebar K_{nv}\setminus K$.
Our main theorem (Theorem \ref{T:n=1 gen-dif}), shows:
\begin{introtheorem}\label{T: intro gen diff}
    Let $\CK$ be an expansion of a real closed field with a weakly o-minimal theory. Let $f: I\to K$ be a definable function. Then, away from finitely many points in $I$, $f'(a)$ exists in $\widebar K_{nv}$: We say that  $f$ is generically $nv$-differentiable.

    In particular, if $\CK$ is definably Cauchy complete (namely has no definable nonvaluational cuts) then every definable function is generically differentiable (in the usual sense).
    \end{introtheorem}
    
In Theorem \ref{T:gen-dif}, we prove a generalization to functions of several variables, asserting that for a definable function $f$ with open domain the Jacobian matrix of partial derivatives at a generic point  provides a $\widebar K_{nv}$-linear approximation of $f$. As a result we show that definable functions satisfy variants of some classical results from differential calculus: A local version of the mean value theorem (Theorem \ref{T: MVT}) and  a function of locally constant derivative is piecewise $K$-affine (Corollary \ref{C:piecewise}).

Towards the proof of the above result we first establish two theorems of independent interest, about definable families of cuts. Recall from \cite{MacMaSt}, that  $\widebar K$, the definable Dedekind completion of $K$, is the collection of all definable cuts in $\CK$. While $\widebar K_{nv}$ is obviously  a subset of $\widebar K$, the latter does not admit a structure of a field or a group. We prove (Theorem \ref{T:no lc into nv v2} and Theorem \ref{T:func into val is lc}):

\begin{introtheorem} 
   Let $\CK$ be an expansion of a real closed field $K$ whose theory is weakly o-minimal.  Let $f:K\to \widebar K$ be a definable function. Then: 
   \begin{enumerate}
       \item If $\Im(f)$ consists of nonvaluational cuts and $f$ is locally constant  then $Im(f)$ is finite.
       \item If $\Im(f)$ consists of valuational cuts then $f$ is locally constant. 
      \end{enumerate}  
\end{introtheorem}
The first clause  generalizes the theorem from \cite{Joh:exchange} on exchange (as the failure of exchange in this setting is equivalent to having a definable locally constant function from $K$ into $K$ with infinite image).

As a corollary of Theorem \ref{T: intro gen diff} we prove (Theorem \ref{T:def fields}):
\begin{introtheorem}
   Let $\CK$ be an expansion of a real closed field whose theory is weakly o-minimal. If $F\sub (\widebar K_{nv})^n$ is a definable field then $F$ is definably isomorphic to either $K$ or its algebraic closure $K(\sqrt{-1})$. In particular, this holds if $F\subseteq K^n$ is a definable field.
   \end{introtheorem}

\subsection*{AI Statement.} This article was prepared without any mathematical assistance from AI.

\section{Preliminaries and Notation}\label{s:PreNot}
We set up some notation and terminology, and review some of the basic facts concerning the main objects of interest in the present paper.  Throughout, structures are denoted by  calligraphic capital letters, $\CM$, $\CN$, $\CK$ etc., and their respective universes by the corresponding Latin letters, $M$, $N$ and $K$. 

Tuples from a structure $\CM$ are always assumed to be finite, and are denoted by small Roman characters $a,b,c,\dots$. We apply the standard model theoretic abuse of notation writing $a\in M$ for $a\in M^{|a|}$. Variables will be denoted $x,y,z,\dots$ with the same conventions as above. We do not distinguish notationally between tuples and variables belonging to different sorts, unless some ambiguity can arise. Capital Roman letters $A,B,C,\ldots$ usually denote small subsets of parameters from $\CM$. As is standard in model theory, we write $Ab$ as a shorthand for $A\cup \{b\}$. In the context of definable groups we will, whenever confusion can arise, distinguish between, e.g., $Agh:=A\cup\{g,h\}$ and $A\,g\!\cdot \! h:=A\cup \{g\!\cdot\! h\}$. 

Given a structure $\CM$ by a "definable set in $\CM$" we mean a subset of $S^n$ for some, \textit{possibly imaginary} sort $S$; definable is with parameters (unless explicitly stated otherwise). By a definable subset of $M^n$ we mean an actual definable subset of a power of the home-sort. At times we will want to emphasize that a set is an interpretable set in $\CM$.

A linearly ordered structure $\CM$ is\textit{ weakly o-minimal} if every definable subset of $M$ is a finite union of convex sets. Throughout all weakly o-minimal structures will be $\aleph_0$-saturated, implying that weak o-minimality is a property of the theory of the structure. Unless specifically stated otherwise, the topology on $M$ is the order topology, and on $M^n$ it is the corresponding product topology. 

In the present paper the ambient weakly o-minimal structure is mostly assumed to be densely ordered.\footnote{This assumption is applied implicitly in references to the literature,  much of which  assumes density of the order underlying  weakly o-minimal structures.} In specific instances we may forgo this assumption and then clearly say so. Indeed, in general, densely ordered weakly o-minimal structures may have non-dense imaginary weakly o-minimal sorts.

Here, a \emph{definable cut in $M$} is a definable, possibly empty, initial segment $C\sub M$ with no maximum. We also identify the cut $C$ with the pair $(C,M\setminus C)$. A definable cut is \emph{rational} if it is realized  (i.e., if it has a supremum in $M$), and \emph{irrational} otherwise. The \emph{definable Dedekind completion} of a weakly o-minimal structure $\CM$, denoted by $\widebar M$, is the collection of all definable cuts. It is, naturally, a linearly ordered structure embedding $(M,<)$, but it is not a definable object (it is, in fact, $\bigvee$-definable), and will not be considered as a model theoretic object.

The usefulness of $\widebar M$ lies in the fact that weak o-minimality implies that for any definable function $f: M\to M$ (or even into $\widebar M$), and for any $a\in M$, the lateral limits (in the order topology) $\lim\limits_{h\to a^\pm }\!f(x)$ exist in $\widebar M$, e.g. $\lim\limits_{x\to a^-}\!f(x)$ is the cut $$\{t\in M: \forall x<a\, \exists x',x<x'<a,\,\,\, t<f(x')\}.$$  By saying that "the limit of $f$ at $x$ exists in $\widebar M$" we will mean 
that the lateral limits at $x$ are equal and belong to $\widebar M$. 


Definable families of cuts will be referred to as \textit{sorts} of cuts. It is sometimes convenient to see them as definable functions from $M^n$ into $\widebar M $. 
If $S$ is a sort of cuts then it is linearly ordered, as a subset of $\widebar M$. Note however that the induced topology and the order topology on $S$ may differ. For example, in a real closed valued field  $K$, with $\CO\sub K$ the valuation subring, consider  the sort $S=\{\sup(x+\CO):x\in K\}\sub \widebar M$ (see further discussion of $\sup$ below).
This is a densely ordered set, but the induced topology is the left-limit topology.

  Assume that  $S\sub S'\sub \widebar M$ are  sorts of cuts. Then $\sup S$ may have different meanings if interpreted in $\widebar M$ or in $S'$. To avoid confusion, let us stress that $\sup(S)$ will always mean $\sup_{\bar{M}}(S)$, namely the cut $\{t\in M:\exists x\in S\,\, t<x\}$. Similarly, for $\inf(S)$.  

Given  $S\sub \widebar M$, we let $\conv(S)$ be \textit{the convex hull of $S$} inside $\widebar M$, namely the (non-definable) set $$\{x\in \widebar M: (\exists s_1, s_2\in S) (s_1\leq x\leq s_2)\}.$$

We recall the definition of valuational cuts in the context of weakly o-minimal ordered groups.

\begin{definition}\label{D: Stab}
Given a cut $\CC=(C_1,C_2)$ in a linearly ordered group $(M,<,+)$, the stabilizer group of $\CC$, also known as \emph{the invariance group of $\CC$} is  $G(C_1)=G(\CC)=\{\delta\in M: \delta+C_1=C_1\}$. 
\end{definition}

The invariance group $G(\CC)$ is a convex (possibly trivial) subgroup of $(M,+)$. It is straightforward to see that $G(\CC)=\{\delta\in M: \forall x\in C_1 \forall y\in C_2\,\, |\delta|<y-x\}$. 

\begin{definition}\label{D: cuts}
	   A cut $\CC=(C_1,C_2)$ is {\em valuational} if $G(\CC)\neq \{0\}$ and {\em nonvaluational}, otherwise.
        
\end{definition}

In the literature, the dichotomy between nonvaluational structures (admitting only definable nonvaluational cuts) and valuational structures (admitting both types of definable cuts) is relatively well studied. Here we study in more depth the dichotomy between the two types of cuts within a single structure. Towards that end it is convenient to introduce: 

\begin{definition}
    Let $\CM$ be a weakly o-minimal expansion of an ordered group. The \textit{definable nonvaluational completion}, denoted $\widebar M_{nv}$ is the subset of $\widebar M$ consisting of all definable nonvaluational (possibly rational) cuts in $M$. 

\end{definition}

\begin{remark}
    Some historical and terminology-related remarks are in order. It seems that nonvaluational cuts in ordered abelian groups were first studied by Baer in \cite{Baer}; these were called by Baer: \emph{Dedekindean cuts}. Cohen and Goffman constructed the nonvaluational completion \cite{CoGo} by means of ``Cauchy'' sequences indexed by ordinals and showed that it is a group (as was done in a general context by Bourbaki, \cite[Chapter II, \S3, no. 7]{Bourbaki}). Independently, Scott constructed the completion for any ordered field by a different method \cite{ScottCompletion}. 

    Note that all their different constructions are for general cuts, i.e. not necessarily definable.
\end{remark}

In general, $\widebar M_{nv}$ is not definable; it is a union of imaginary sorts and if $S$ is a sort of nonvaluational cuts then there is some sort $S'$ (consisting of nonvaluational cuts) such that for any $a,b\in S$, $a+b:=\sup\{a'+b': a'<a, b'<b\}\in S'$. We similarly define $a\cdot b$, when $M$ expands a field. 
While $\widebar M_{nv}$ is not a definable object, the restriction of the field operations to any given sorts is definable.

We note the following:

\begin{fact}\label{F: M*}
    Let $\CM$ be a weakly o-minimal expansion of a linearly ordered group. 
    \begin{enumerate}
        \item $(M,<)$ naturally embeds in $\widebar (\widebar M_{nv},<)$ as a dense subset. 
        \item The group operations on $M$ extend continuously to $\widebar M_{nv}$. 
        \item If $M$ expands a (real closed) field 
        then $\widebar M_{nv}$ is an ordered field extension of $M$ and it is a real closed field. 
    \end{enumerate}
\end{fact}
\begin{proof}
    Let $\widehat M$ be the Cauchy completion of $M$ as defined in \cite{CoGo}. (1) and (2) follow from \cite[Section 2]{CoGo}, once we restrict from $\widehat M$ to $\widebar M_{nv}$. Although the authors do not prove that $M$ is dense in its Cauchy completion, it is standard to verify and noted in the historical remarks of \cite{ScottCompletion}

    For (3), it is convenient to use the results in \cite{ScottCompletion}, namely Theorems 1 and 2 there. $M$ is dense in $\widehat M$ and as $M$ is real closed then so is $\widehat M$. To see that $\widebar M_{nv}$ is real closed, note that if $p(x)$ is a polynomial over $ \widebar M_{nv}$ and $a\in \widehat M$ is a root of $p$ then, by taking derivatives, we may replace $p$ by a polynomial $q(x)$ over $\widebar M_{nv}$ which vanishes at $a$, while its derivative does not.
But then, the cut $a$ is an endpoint of one of the convex components of $\{x\in M: q(x)>0\}$, thus it is a definable cut, so $a\in \widebar M_{nv}$. It follows that $\widebar M_{nv}$ is also real closed. 
\end{proof}

The monotonicity theorem is an important tool in the study of weakly o-minimal theories. First proved for densely ordered weakly o-minimal theories in \cite[Theorem 3.3] {MacMaSt} and recently generalized to general weakly o-minimal theories in \cite[Theorem 3]{MocTan}. 

\begin{fact}\label{F:monotonicity}\cite[Theorem 3]{MocTan}
    Let $M$ be weakly o-minimal but not necessarily dense and $f: M\to D$ be a definable partial function, where $D$ is any definable linear order. Then away from finitely many points $f$ is locally constant, locally strictly increasing or locally strictly decreasing.
\end{fact}

While a sort $S\sub \widebar M$ need not be weakly o-minimal (indeed, it need not even be dp-minimal), we observe the following: 
\begin{lemma}\label{L: images are WOM}
    Let $f: M\to \widebar M$ be a definable partial function, with $\mathrm{Im}(f)$ infinite. 
    \begin{enumerate}
        \item $\mathrm{Im}(f)$ is weakly o-minimal (possibly, not densely ordered), with respect to some definable linear order.
        \item There exists a finite subset $F\subseteq \mathrm{Im}(f)$ such that, locally around each point of $\mathrm{Im}(f)\setminus F$, this order agrees with either the $\widebar M$-induced order  or with its inverse. 
        \item There is an interval $I_0\sub \dom(f)$ such that $f\rest I_0$ is weakly monotone with convex fibers and $f(I_0)$  infinite.
    \end{enumerate} 
\end{lemma}
\begin{proof}
    Let $I$ be the domain of $f$. Let $I_0\subseteq I$ consist of all those $x\in I$ such that $x$ belongs to the leftmost convex component of $f^{-1}(f(x))$. Then $f(I_0)=f(I)$ and the fibers of $f\rest I_0$ are convex. If $E$ is the associated (convex) equivalence relation $xEy \Leftrightarrow f(x)=f(y)$, then $f(I)=f(I_0)$ is in definable bijection with $I_0/E$, inducing a linear order on it, whose weak o-minimality follows from that of $\CM$. Let $\trianglelefteq$ be the order which $I_0/E$ induces on $\mathrm{Im}(f)$ and let $\leq$ be the order induced from $\widebar M$.

    Consider the identity map $(\mathrm{Im}(f),\trianglelefteq)\to (\mathrm{Im}(f),\leq)$. By Fact \ref{F:monotonicity} applied to the weakly o-minimal structure $(\mathrm{Im}(f),\trianglelefteq)$, this map, away from finitely many points, is locally strictly increasing or decreasing (it cannot be constant by injectivity). The result follows.
\end{proof}


    In Lemma \ref{L: images are WOM}(1), we use the notion of a weakly o-minimal definable set. A small remark is required since this definable set might not be stably embedded.
    
	\begin{remark}
		Recall that given an $\aleph_1$-saturated structure $\CM$, a definable set $S \sub M^n$  is
		{\em stably embedded} if for every definable $X \sub  M^ {nk}$  the set $X \cap S^k$  is definable using only parameters from $S$. Thus, in this case the language of the $\CM$-induced structure on $S$ does not change when we move to elementary extensions.

		Our interest below is with definable sets $D$ with an induced linear ordering with respect to which the $\CM$-induced structure  is weakly o-minimal. 
		However, weakly o-minimal definable sets need not be stably embedded. In fact, even if $\CM$ itself is weakly o-minimal and  $D$ is an interval (thus a weakly o-minimal sort), $D$ is not necessarily stably embedded.  Thus, without specifying the ambient parameter set,  the notion of  ``the structure induced on $D$'' is not well-defined. For $A \sub M$  we let $\mathcal D_A$ denote the structure obtained by taking the traces on $D^n$,  $n \in  \Nn$, of all $A$-definable sets. 
		
		When we say that $D$ is weakly o-minimal, we will  further assume that for any $A\sub M$,  the structure $\CD_A$  has a weakly o-minimal theory. Because $\CM$ is assumed to be sufficiently saturated, if $\CN\succ \CM$ then the same remains true for any $A\sub N$. 		
\end{remark}

We fix some useful notation. For $c\in \widebar M$, let $c^+$ be the type
$
c^+:=\{c<x<s: s>c\,,\, s\in M\}, 
$  and similarly define  $c^-$. If $c=\infty$ we let $c^-:=\{s<x: s
\in M\}$,  and analogously  $c=-\infty$. By weak o-minimality, these are complete types. As every complete type defines in elementary extensions a convex set, the following is easy: 
\begin{lemma}\label{L: types are conex}
    If $f: (a,b)\to K$ is definable then $f(a^+)$ is convex, i.e., the set of realizations in an elementary extension is convex.
\end{lemma}

At times, we would like to restrict the parameters. For instance, for $\CN\prec \CM$ we will denote by $(c/N)^+$ the type $\{c<x<s: s>c\,, \,s\in N\}$.

\section{Some basic properties of weakly o-minimal theories}

    In the present section we introduce some of the tools needed later on. 
    \vspace{.2cm} 
		
	\textbf{Throughout this section $\CM$ denotes an $\aleph_0$-saturated densely ordered weakly o-minimal structure.} 
    
    Let $X\subseteq M^n$ be a nonempty definable subset, $a\in M^n$. The topological dimension, denoted $\dim(X)$, is the maximal $k$ for which there is a coordinate projection $\pi:M^n\to M^k$ such that $\pi(X)$ has non-empty interior. For a tuple $a$ we set $\dim(a/C):=\min\{ \dim(X): X\in \tp(a/C)\}$. It can be shown, and we will not show it here, that this notion of dimension is equal to other notions such as dp-rank. For example, this equality is shown in \cite{SimWal} when $\CM$ expands an ordered abelian group.

\begin{definition}
    Let $X\subseteq M^n$ be an $A$-definable set, $a\in X$ and $B\supseteq A$. The point $a$ is $B$-\emph{generic}  in $X$ (or, {\em generic in $X$ over $B$}) if $\dim(a/B)=\dim(X)$.
\end{definition}

	\subsection{The independent neighborhood property and the monotonicity theorem}
    The independent neighborhood property  allows us to zoom in on generic points without harming the genericity.

    \begin{fact}\cite[Theorem A.15, Lemma 3.2]{CasHasJoh}\label{F:inp}
        Let $A\sub M$ be a parameter set and $I\sub M^n$ an open box 
		definable over arbitrary parameters. For any $a\in I$ and $b\in M^k$,  there exists $C\supseteq A$ and a $C$-definable open box $J\sub M^n$, with $a\in J\subseteq I$, such that $\dim(a,b/A)=\dim(a,b/C)$.
    \end{fact}

    Below, when referring to this property we will just write INP. \\

As a demonstration of the strength of INP, we give below, in Corollary \ref{C:montonicity}, a short proof of a version the monotonicity theorem, Fact \ref{F:monotonicity}.
\\

We need a ``Definable Ramsey" result for weakly o-minimal structures.

\begin{lemma} \label{Ramsey}
    Assume that $a\in M$ is generic over $A$, $\varphi(x,y)$,  is an $A$-formula, $|x|=|y|=1$, and assume that there is $b>a$ for which $\varphi(a,a_1)$ holds for all $a<a_1<b$. Then there is an interval $I\ni a$ such that for all $a_1<a_2$ in $I$, $\varphi(a_1,a_2)$.
\end{lemma}
\begin{proof}
    By INP, there is $a<b'<b$, with $\dim(a/b'A)=1$ such that $(\forall x)(x\in (a,b') \to \varphi(a,x))$. Since $\dim(a/Ab')=1$  there is an interval $I\ni a$ such that for all $a_1\in I$ the following holds: $(\forall x)(x\in (a_1,b')\to \varphi(a_1,x))$. By shrinking $I$, we may assume that $\sup(I)\leq b'$ and therefore for all $a_1<a_2$ in $I$, $\varphi(a_1,a_2)$.
\end{proof}

\begin{remark}
    The above result does not extend to ternary relations, even in o-minimal structures. As an example consider $R(x,y,z)$ the relation given by $y<z<x+|y-x|$. There is no interval $I$ such that either all $x<y<z$ in $I$ satisfy $R$ or all $x<y<z<$ fail to satisfy $R$. 
\end{remark}

As was noted in \cite[Example A.5]{Joh:exchange}, Lemma \ref{Ramsey} immediately implies:
\begin{corollary}\label{C:montonicity}
    Let $f: M\to (Y,\vartriangleleft)$ be a definable partial function into a definable linearly ordered set. Then, away from finitely many points, $f$ is locally constant, or locally strictly  monotone. 
\end{corollary}
    


Finally, we  need the following variant     of INP, which is an extension of \cite[Lemma A.7]{Joh:exchange}.

\begin{lemma}\label{L: Johnson A7.2}
Assume that $f:M\to S$ is a partial definable function, where $S$ is any definable set, and $\varphi(x,e)$ a formula, such that for all $0<\epsilon_1<\epsilon_2$, and $x\in M$, if $\varphi(x,\epsilon_2)$ then $\varphi(x,\epsilon_1)$.

Assume that $f$ and $\varphi$ are all defined over $A$ and $a\in \dom(f)$ is such that  $f(a)\notin \acl(A)$.

If  $\varphi(a,\epsilon)$ holds for some $\epsilon>0$ then there exists  $\epsilon_0>0$ and a convex $C\ni a$, defined over $A\epsilon_0$,  such that $f(a)\notin \acl(A\epsilon_0)$ (so $f(C)$ is infinite) and for all $x\in C$,   $\varphi(x,\epsilon_0)$ holds. 
\end{lemma} 
\begin{proof}
Let $\kappa> |T|+|A|$ be a large enough cardinal (but smaller than the saturation of $M$). Since $f(a) \notin \acl(A)$, we can find $(a_i)_{i<\kappa} \ldots$ realizing $\tp(a/A)$ with the $f(a_i)$ pairwise distinct. For each $i\geq 1$, there is $\epsilon_i > 0$ such that $\phi(a_i,\epsilon_i)$ holds. By saturation, there exists  $\epsilon_0 > 0$ with $\epsilon_0 \le \epsilon_i$ for every $i<\kappa$. Then $\phi(a_i,\epsilon_0)$ holds for every $i$. Equivalently, $X:=\{x\in M: \CM\models \phi(x,\epsilon_0)\}$ contains all the $a_i$. 

Take $C_0$ a convex component of $X$ containing $\kappa$-many of the $a_i$. 
Choose some $a_i$ in $C_0$ with $f(a_i)\notin \acl(A\epsilon_0)$. We have a convex $A\epsilon_0$-definable set $C_0$ containing $a_i$ such that $f(a_i)\notin \acl(A\epsilon_0)$ and $\phi(x,\epsilon_0)$ holds for every $x \in C_0$. Take $\sigma$ an automorphism over $A$ moving $a_i$ to $a$ and let $\epsilon = \sigma(\epsilon_0)$ and $C = \sigma(C_0)$. Then $C$ is a convex $A\epsilon$-definable set containing $a$ such that $f(a)\notin \acl(A\epsilon_0)$ and $\phi(x,\epsilon)$ holds for every $x \in C$.
\end{proof}

\subsection{Density of nonvaluational cuts}

For this section let $\CM$ be a sufficiently saturated weakly o-minimal expansion of an ordered abelian group.

We need some general results on nonvaluational cuts.

\begin{lemma}\label{L:holes between nv}
    Let $S$ be an infinite sort of nonvaluational cuts. Then there is no strictly monotone sequence $(c_n)_{n<\omega}$ of elements in $S$ such that for all $n$, $\conv(c_n,c_{n+1})\nsubseteq \cl_{\widebar M}(S)$.
   
\end{lemma}
\begin{proof} 
    Assume such a sequence exists and that the sequence is, say, decreasing. Thus, there are non-empty intervals $I_n=(u_n,v_n)\subseteq M$ lying between $c_n$ and $c_{n+1}$ such that $\conv(I_n)\cap S=\emptyset$ for all $n$.  By saturation, there is some $\epsilon>0$, such that $0<\epsilon<v_n-u_n$ for all $n$.
    
    Let $D$ be the definable set $\{x\in M: \exists s\in S,\, x<s<x+\epsilon\}$. Since each $c_n$ is nonvaluational, $D$ intersects nontrivially some neighborhood of $c_n$. On the other hand, for each $n$, $(u_n,v_n-\epsilon)\cap D=\emptyset$. This implies that $D$ has infinitely many convex components, contradicting weak o-minimality.
\end{proof}

\begin{corollary}\label{C:partition of nv}
    Assume that $S\sub \widebar M_{nv}$ is a definable sort of nonvaluational cuts. Then $S$ can be partitioned into finitely many definable sets $S_i$, such that for each $i$, $\conv(S_i)\cap M\sub \cl(S_i)$ (where the closure is taken in $\widebar M$). 
\end{corollary}
\begin{proof}
    Define an equivalence relation $\sim$ on $S$: $a\sim b$ if $S$ is  dense in the $\widebar M$-interval between $a$ and $b$. Note that this is a definable equivalence relation with convex classes. By Lemma \ref{L:holes between nv}, there are only finitely many equivalence classes. Let $S_1,\dots, S_n$ be the classes. It is straightforward to see they satisfy the requirement.
%
\end{proof}

\begin{corollary}\label{C:interval inside closure}
    Let $S$ be an infinite sort of positive nonvaluational cuts with $0\in \cl(S)$.  Then there is $d>0$ such that $[0,d]\subseteq \cl(S)$.
\end{corollary}
\begin{proof}
By Corollary \ref{C:partition of nv}, we may assume that $\conv(S) \cap M$ is contained in $\cl(S)$.  Since $0 \in \cl(S)$ and $S$ consists of positive elements, there are elements of $S$ between $0$ and $\epsilon$, for any positive $\epsilon \in M$.  Thus $\conv(S) \cap M$ contains some open interval $(0,d)$, implying that $[0,d]$ is contained in $\cl(S)$.
\end{proof}

\begin{corollary}\label{C:nv are dense}
     Let $S$ be an infinite sort of nonvaluational cuts. Then, for all but finitely many $a\in S$,
     $a=\inf_{\widebar M}\{s\in S:s>a\}=\sup_{\widebar M}\{s\in S:s<a\}.$

     In particular, there is a finite set $F\sub S$, such that 
     on $S\setminus F$ the order is dense, and the order topology coincides with the $\widebar M$-induced topology on $S\setminus F$. 
\end{corollary}
\begin{proof}
By Lemma \ref{L:holes between nv}, there is a partition $S=\bigcup S_i$, such that each $S_i$ is dense in $\conv(S_i)$. If $a\in S$ is neither a maximum nor a minimum of one of the $S_i$ then it satisfies the claim. 

After removing these finitely many points, the set $S$ is  densely ordered and its order topology agrees with the $\widebar M$-topology.
\end{proof}

A noteworthy corollary of the above lemma is that if $E$ is a definable convex equivalence relation on $M$ then, with finitely many exceptions, the classes of $E$ must be bounded, above and below, by valuational cuts. This gives an alternative proof of the (known) fact that nonvaluational weakly o-minimal ordered groups have Exchange. 

    \section{Definable maps into cuts}
       
As already noted in \cite{MacMaSt}, definable functions into $\widebar M$ play a crucial role in the study of weakly o-minimal structures (e.g., full cell decomposition is formulated in terms of functions into $\widebar M$). In the present section, we prove some fundamental structural properties of such functions in the context of weakly o-minimal expansions of fields. 

\textbf{Throughout this section, let $\CK$ be a sufficiently saturated weakly o-minimal expansion of an ordered field $K$.} \\

The following was recently proved by the fourth author: 
        

\begin{fact}\cite{Joh:exchange}\label{F:exchange}
    If $\CK$ is a weakly o-minimal expansion of a field then $\dcl$ satisfies exchange. Equivalently,\footnote{\cite[Theorem 5.4]{Kulp}} if $f:K\to K$ is a definable locally constant function then $Im(f)$ is finite.
\end{fact}

The following standard definition is important for everything to come.

 \begin{notation}\label{N:divided difference}
    For a  function $f: K\to K$ or $f: K \to \widebar K_{nv}$, the \emph{divided difference of $f$} is the function $m_f(x,y):=\frac{f(x)-f(y)}{x-y}$, defined on pairs of \emph{distinct} elements $x,y\in \dom(f)$.  When no confusion can arise we write $m(x,y)$ for $m_f(x,y)$. 
\end{notation}

\subsection{Definable functions into nonvaluational cuts}
    Recall that $\widebar K_{nv}$ is the set of all nonvaluational cuts in $\widebar K$ (including rational cuts identified with the points of $K$); it is a real closed field (Fact \ref{F: M*}).

    In this section we extend Fact \ref{F:exchange}:


\begin{theorem}\label{T:no lc into nv v2}
    A definable locally constant unary function $f:I\to \widebar K_{nv}$ has finite image.
\end{theorem}
\begin{proof}
    Assume towards a contradiction that $f: I\to \widebar K_{nv}$ is a counterexample, i.e.,  a definable locally constant function with infinite image $S$. 
    
      There is no harm in assuming 
      that the fibers of $f$ are convex, that $f$ is weakly increasing,  and that $S$ is weakly o-minimal with the induced order (Lemma \ref{L: images are WOM}(3)). 
      Applying Corollary \ref{C:partition of nv}, we may assume that $S$ is dense in $\conv(S)$, and in particular, $(S,<)$ is densely ordered.
      
   For simplicity, assume everything is defined over $\emptyset$. Fix $a\in I$ with $f(a)$ generic in $S$ (this can be done since $S$ is weakly o-minimal).  Let $\CN\prec \CK$ be a sufficiently saturated small model containing $a$. 
   For $b\in \widebar N_{nv}$, let $(b/N)^+$ denote the positive infinitesimal type of $b$ over $N$, namely the collection of formulas in the variable $x\in M$: $$\{b<x<c: c\in N, c>b\}.$$

     For $s\in S$ with $s>f(a)$ and $x\in I$, let  $$D^s(x):=\{m(x,y):y\in I \text{ with } f(x)<f(y)<s\}.$$ 

\begin{lemma}\label{L: 2.1}

 For every $s\in S$ with $s>f(a)$, there is an open $J\ni f(a)$  and $\epsilon>0$, such that for all $x\in f^{-1}(J)$,   $[0,\epsilon]\sub \cl(D^s(x))$.
\end{lemma}
\begin{proof} 
\renewcommand*\qedsymbol{$\square$ (lemma)}
Fix $s>f(a)$. By INP, applied to $f(a)\in S$, we may assume that $f(a)$ is still generic in $S$ over $s$. 
    The following observation is similar to \cite[Lemma A.15]{Joh:exchange}: 
    \begin{claim}
         $0\in\cl(D^s(a))$.
    \end{claim}
    \begin{claimproof}
        Let $\epsilon>0$. Let $c>a$ be such that $f$ is constant on $[a,a+c]$. Since $S$ is dense in $\conv(S)$, we can find $y\in I$ such that $f(a)<f(y)<\min(s,f(a)+c\epsilon)$. As $f$ is constant on $[a,a+c]$, necessarily $y-a>c$. Thus $m(a,y)=\frac{f(y)-f(a)}{y-a}<\frac{c\epsilon}{y-a}<\epsilon$, and $m(a,y)\in D^s(a)$.
    \end{claimproof}
By Corollary \ref{C:interval inside closure},  we conclude that there exists $\epsilon>0$, such that $[0,\epsilon]\sub \cl(D^s(a))$.

Since $f(a)\notin \acl(s)$, we may apply  Lemma \ref{L: Johnson A7.2} and find an $\epsilon_0>0$ and a definable convex set $C\ni a$ such that $f(a)\notin \acl(s\epsilon_0)$ and for all $x\in C$, we have $[0,\epsilon_0)\in \cl(D^s(x))$.  Because $f(a) \notin \acl(s\epsilon_0)$ it follows that $f(C)$ is infinite and $f(a)$ is in the interior of $f(C)$. Since each fiber of $f$ is assumed to be convex, we can now find an open $J\sub f(C)$ such that $f^{-1}(J)\sub C$ and $f(a) \in J$.
\end{proof}

   \begin{claim}\label{C: technical}
       For all $x$ such that  $f(x)\models (f(a)/N)^+$, all $d\models (0/N)^+$ and $0\leq c<d$, there is $y>x$ with $f(y)\models (f(a)/N)^+$, such that $c<m(x,y)<d$.
   \end{claim}
   \begin{claimproof}
   Fix $x$ so that $f(x)\models (f(a)/N)^+$, $c< d$ and $d\models (0/N)^+$. By compactness, it is sufficient to see that for every $s>f(a)$, $s\in N$, there exists $y$, $f(a)<f(y)<s$, such that $m(x,y)\in (c,d)$. 
   
   Fix $s>f(a)$ in $N$. 
By Lemma \ref{L: 2.1}, applied in $\CN$,  there exists an open $N$-interval $J\ni f(a)$ and $\epsilon>0$ in $N$, such that for all $x\in f^{-1}(J)$, $[0,\epsilon]\sub \cl(D^s(x))$.

By our assumptions on $x,c,d$,  we have $f(x)\in J$  and  $(c,d)\sub [0,\epsilon]$. Hence there is $y$, $f(x)<f(y)<s$ such that $m(x,y)\in (c,d)$. \end{claimproof}

Our goal is to find an infinite strictly increasing sequence $(x_n)_{n\geq 0}$ in $I$, such that for all $n$, $f(a_n)\models (f(a)/N)^+$, and \begin{equation}\label{eq: sequence} 
    m(x_0,x_2)<m(x_0,x_4)<m(x_0,x_6)<\cdots \,\,\,\, \cdots <m(x_0,x_5)<m(x_0,x_3)<m(x_0,x_1)\end{equation}

We shall be using the following standard fact, true in any ordered field: If $x_0<x_1<x_2$ then for any function $f$ on these points, 
\begin{gather} \label{eq: properties} 
    m(x_1,x_2)<m(x_0,x_1)\Longrightarrow m(x_1,x_2)<m(x_0,x_2)<m(x_0,x_1) \text{ and }\\
    m(x_0,x_1)<m(x_1,x_2)\Longrightarrow m(x_0,x_1)<m(x_0,x_2)<m(x_1,x_2).
\end{gather}

We choose the first elements as follows:
Fix $x_0\in I$ with $f(x_0)\models (f(a)/N)^+$. Claim \ref{C: technical} supplies $x_1>x_0$ with $f(x_1)\models (f(a)/N)^+$ such that $m(x_0, x_1)\models (0/N)^+$. Next, we apply the claim to $f(x_1)\models (f(a)/N)^+$ and $0<m(x_0,x_1)\models (0/N)^+$ to get $x_2>x_1$, $f(x_2)\models (f(a)/N)^+$, with $0<m(x_1,x_2)<m(x_0,x_1)$. By Equation (\ref{eq: properties}),   $0<m(x_0,x_2)<m(x_0,x_1)$.
The rest of the sequence (\ref{eq: sequence}) is obtained by repeatedly applying Claim \ref{C: technical}, together with Equations (2),(3).

Given such a sequence, by saturation, there is some element $d\in K$ with $m(x_{2n+2},x_0)<d<m(x_{2n+1},x_0)$ for all $n\geq 0$. Let $\ell(x)=d(x-x_0)+f(x_0)$ be the $\widebar K_{nv}$-line with slope $d$ through $(x_0,f(x_0))$. 

The definable set $\{x\in K: f(x)<\ell(x)\}$ contains all $x_{2n+2}$, for $n\geq 0$, and misses all $x_{2n+1}$, for $n\geq 0$, contradicting weak o-minimality.   This ends the proof of Theorem \ref{T:no lc into nv v2}.
\end{proof}

\begin{question} \label{Q: Is exchange enough 1} Theorem \ref{T:no lc into nv v2} clearly implies the exchange principle. Thus, it is natural to ask whether it is actually equivalent to it. Namely, is it true, under exchange alone (without assuming an underlying field) that every definable locally constant function into $\widebar M_{nv}$ has a finite image. 
\end{question}

\subsection{Definable functions into valuational cuts}

While, by Theorem \ref{T:no lc into nv v2}, in a weakly o-minimal field, functions into nonvaluational cuts are not locally constant (unless their image is finite), the aim of this section is to prove that functions into valuational cuts are always  locally constant. 

 We start working in a weakly o-minimal expansion of a group $\CM$. Later, we restrict to fields.
We keep the notation we used before:  $\widebar M$ is the collection of all definable cuts in $\CM$ and for  $S\sub \widebar M$, we let $\conv(S)$ denote the convex hull of $S$ in $\widebar M$. The proof is based on the following generalization of a compositional inverse, introduced in \cite{Joh:exchange}.

\def\gf{f_\downarrow}

\begin{definition}
    Given a definable, partial weakly increasing function $f:M\to \widebar M$, let $$a_f=\inf \Im(f)\,\,;\,\, b_f=\sup \Im(f),$$ and define $\gf:\conv(a_f,b_f)\cap M\to \widebar M$ to be the partial function defined by 
    $$
    \gf(y)=\sup\{x\in \dom(f): f(x)<y\}=\sup f^{-1}(\conv(-\infty, y)).$$ 
    

\end{definition}


By definition, $\gf$ is weakly increasing. Since $f$ takes values in $\widebar M$, $\gf$ is not its inverse, but if $f:(a,b)\to M$ is strictly increasing and continuous then $\gf=f^{-1}$, so $(\gf)_\downarrow=f$. Here are some more examples:

\begin{example}\label{Example: gf}

\begin{enumerate}
\item If $f\equiv c<\infty$ is a constant partial function then $\dom(\gf)=\emptyset$.

If $f$ takes $n$-many values, $n>1$, then $\gf$ takes $n-1$ values.


\item Assume that $f:M\to M$ is locally constant, weakly increasing, with convex infinite image. Then $\gf(y)=\inf f^{-1}(y)$ and $(\gf)_\downarrow(x)=f(x)$.

    \item  
    Assume that $\alpha:M\to M$ is a locally constant, weakly increasing function with convex infinite image. Define $f:M\to \widebar M$ by $f(x)=\sup \alpha^{-1}(x)$. Then, $\gf(y)=\alpha(y)$ and $(\gf)_\downarrow(x)=\inf \alpha^{-1}(x)\neq f(x).$

\end{enumerate}
\end{example}

The strategy is the following: Assume that $f:M\to \widebar M$ is a function into a valuational sort: We first show that for such an $f$,  the function $\gf$ is generically locally constant. By Theorem \ref{T:no lc into nv v2}, it follows that $\Im(\gf)$ is valuational except for finitely many values. Thus, applying the same reasoning to $\gf$ we conclude that $(\gf)_\downarrow$ is also generically locally constant. Finally, we show that if $f$ were strictly increasing then, generically, $(\gf)\downarrow=f$, reaching a contradiction (since $(\gf)_\downarrow$ must be locally constant). 

We first need the lemma below. Given $S\sub \widebar M$ and $d\in \widebar M$, let $S_{<d}=\{x\in S:x<d\}$.
\begin{lemma} \label{L: gaps in image} Assume that $\CM$ satisfies exchange.
If $f:M\to \widebar M$ is strictly increasing then for all but finitely many $d\in \Im(f)$, $d=\sup_{\widebar M}(\Im(f)_{<d})$
    
\end{lemma}
\begin{proof} 
Assume towards contradiction that there are infinitely many $d\in \Im(f)$ such that $d>\sup_{\widebar M}(\Im(f)_{<d})$, and for each such $d$, let $Y_d=\{y\in M: \sup_{\overline{\bar M}}(\Im(f)_{<d})<y<d\}$. Each $Y_d$ is infinite and the $Y_d$ are pairwise disjoint. Let $Y$ be the union of all these $Y_d$ and define $s:Y\to M$ to be the function which sends each $Y_d$ to $f^{-1}(d)$ (using the fact that $f$ is injective). It is easy to verify that $s$ is a definable, locally constant function with infinite image, contradicting exchange.
\end{proof}

We shall need the following properties of $\gf$:

\begin{lemma}\label{L: properties of gf}
Assume that $f:M\to \widebar M$ is a partial weakly increasing function. 
\begin{enumerate}
    \item Assume that $y_1,y_2\in \dom(\gf)$, and $y_1<f(x_1)<f(x_2)<y_2$, for $x_1,x_2\in \dom(f)$. Then $\gf(y_1)<\gf(y_2)$.
    \item Assume that $\Im(f)$ has no minimum and no maximum. Then $$\inf \Im(\gf)=\inf\dom(f) \mbox{ and } \sup \Im(\gf) =\sup \dom (f),$$ so $\conv(Im \gf)=\conv(\dom f)$.
\end{enumerate}
\end{lemma}
\begin{proof} (1) This follows directly from the definition of $\gf$. (2) follows from (1) together with the definition of $\gf$.
    \end{proof}

Next we prove:
\begin{lemma}\label{L: valuational implies loc constant}
    If $f:M\to \widebar M$ is a partial weakly increasing function whose image is a valuational sort then outside a finite subset of $\dom(\gf)$, $\gf$ is locally constant.
\end{lemma}
\begin{proof} It is enough to prove that there exists an interval $I\sub \dom(\gf)$ on which $\gf$ is constant. By definition, each valuational cut is translation-invariant under some $\epsilon>0$. Thus, by  \cite[A.7]{Joh:exchange} there is $\epsilon>0$ and an interval $I\sub \dom(f)$ such that $f(I)$ is infinite and for all $x\in I$,  $f(x)+\epsilon=f(x)$, namely the cut $f(x)$ is invariant under $\epsilon$. By replacing $f$ with $f\rest I$, we may assume that this is true for all $x\in \dom(f)$.

But then, for every $y\in \dom(g_{f})$, $(f)^{-1}(\conv(-\infty,y))=(f)^{-1}(\conv(-\infty,y+\epsilon))$. It follows that
$g_{f}(y)=g_{f}(y+\epsilon)$. The same is true for all $e_1<\epsilon$, hence $g_{f}$ is locally constant on $\conv(f(I))\cap M$.
    \end{proof}

\begin{lemma}\label{L: double gf}
    Assume that $\CM$ satisfies exchange, $a<b$ in $M$, and that $f:(a,b)\to \widebar M$ is a definable,  strictly increasing function. Then $\dom((\gf)_\downarrow)=\dom(f)$ and  for all but finitely many $x\in \dom(f)$, $(\gf)_\downarrow(x)=f(x)$.

\end{lemma}
\begin{proof}

Since $f$ is strictly increasing, $\Im(f)$
has no maximum or minimum, so  by Lemma \ref{L: properties of gf}, $\conv(\Im(\gf))=\conv(a,b)$. By definition, $\dom((\gf)_\downarrow)=(a,b)=\dom(f)$.  By Lemma \ref{L: gaps in image}, we may assume, after removing a finite subset, that for all $x\in (a,b)$, $$f(x)=\sup(\Im(f)_{<f(x)}).$$ 

 We  claim that for $x\in \dom(f)$, $(\gf)_\downarrow(x)=f(x)$. 
Indeed, take $x\in (a,b)$ and let $h=f\rest(a,x)$. 

By definition,  if $a_h=\inf \Im(h)$ and $b_h=\sup \Im(h)$ then
$$\dom(h_\downarrow)=\conv(a_h,b_h) \cap M.$$

Applying again, Lemma \ref{L: properties of gf} to $h$ and $h_\downarrow$, we conclude that $\conv(\Im(h_\downarrow))=\conv(a,x)$, and therefore $(h_\downarrow)^{-1}(\conv(a,x))=(\gf)^{-1}(\conv(a,x))=\conv(\dom(h_\downarrow))$. 

Finally, our assumption on $\Im(f)$, $\sup \Im(h)=\sup \Im(f)_{<f(x)}=f(x)$, and therefore 
$$(\gf)_\downarrow(x)=\sup\, (\gf)^{-1}(\conv(-\infty,x))=\sup \Im(h)=f(x),$$
which ends the proof.
\end{proof}

We can now prove:
\begin{theorem}\label{T:func into val is lc}
  Assume that $\CK$ expands a field. Then every definable partial function $f:K\to \widebar K$ whose image consists of valuational cuts must be locally constant.
\end{theorem}
\begin{proof} It is sufficient to prove that there exists an interval in $\dom(f)$ on which $f$ is constant.
 By  \cite[Theorem 3.3]{MacMaSt} (or Corollary \ref{C:montonicity}), $f$ is weakly monotone in a neighborhood of every $x\in \dom(f)$ outside a finite set. Thus, we may assume, towards a contradiction, that it is, locally, strictly monotone in a neighborhood of every point in $\dom(f)$. Hence, without loss of generality, $f$ is strictly increasing on its domain. 

By Lemma \ref{L: valuational implies loc constant}, the function $\gf$ is generically locally constant. By Theorem \ref{T:no lc into nv v2},  every $\gf(y)$ is valuational except for finitely many values. Thus, by Lemma \ref{L: valuational implies loc constant}, applied to $\gf$, we conclude that $(\gf)_\downarrow$ is also generically locally constant. By Lemma \ref{L: double gf}, we have generically, $(\gf)_\downarrow=f$. This is a contradiction since, on one hand, $(\gf)_\downarrow$ is locally constant and on the other $f$ is strictly increasing.
\end{proof}

Recall that, in the terminology of \cite[\S 9]{HaHaPeGps} two sorts $S_1, S_2$ are \emph{foreign} if there is no definable finite-to-finite correspondence $C \sub  X_1 \times  X_2$ , where $X_i \sub S_i^{n_i}$ are definable infinite subsets. By \cite[Lemma 9.2]{HaHaPeGps}, it is enough to consider the case when $n_1=n_2=1$. If $S_1, S_2$ are sorts in a weakly o-minimal structure then as $\dcl^{eq}=\acl^{eq}$ (\cite[Lemma 2.16]{MelRCVFEOI}), this is equivalent to the nonexistence  of any definable bijection between $f: X_1\to X_2$ with $X_i\sub S_i$ infinite. The results of the present section show that, in a weakly o-minimal field, valuational and nonvaluational sorts are foreign. To show this we need the following:

\begin{lemma}\label{L: F in nv is lsi to K}
Let $\CK$ be a weakly o-minimal expansion of a field. Let $S$ be an infinite nonvaluational sort. There exists an infinite definable subset $S'\subseteq S$ and a definable injection $S'\hookrightarrow K$.    
\end{lemma}
\begin{proof}
    Let $S$  be an infinite interpretable set consisting of nonvaluational cuts. By \cite[Lemma 5.1]{HaHaPeVF}, there is a finite-to-finite correspondence between an infinite definable subset of $S$ and an infinite definable subset of $K/E$ for some definable equivalence relation on $K$. Since in weakly o-minimal theories $\dcl^{\eq}=\acl^{\eq}$ (\cite[Lemma 2.16]{MelRCVFEOI}), this implies that there exists  a definable bijection between an infinite definable  $S'\sub S$ and a definable subset of $K/E$. As $E$ can be assumed to be a convex equivalence relation, we get a partial function from $K$ onto $S'$. If this function is injective, we are done. Otherwise it is locally constant, contradicting Theorem \ref{T:no lc into nv v2}.
\end{proof}

\begin{corollary}
    Let $\CK$ be a weakly o-minimal field. Let $S_1$ be a nonvaluational sort, $S_2$ a valuational sort. Then $S_1$ and $S_2$ are foreign. 
\end{corollary}
\begin{proof}
    Assume not. Then by the above discussion there is a definable injection $f: S_1\to S_2$. By Lemma \ref{L: F in nv is lsi to K} this implies that there is a definable (partial) injection $g: K\to S_2$ with infinite domain, contradicting Theorem \ref{T:func into val is lc}. 
\end{proof}

We do not know whether the following weaker property may be true: 
\begin{question}
 Assume that $f:K\to K$ is a definable continuous, strictly monotone partial map, defined on $a^-$, for some $a\in K$.  Is the cut $\lim\limits_{t\to a^-}f(t)$ necessarily nonvaluational? 
 \end{question}

 As in Question \ref{Q: Is exchange enough 1}, we may ask here:
\begin{question}
    Assume that $\CM$ is a weakly o-minimal structure satisfying exchange. Is it true that every definable map into a valuational sort is locally constant?
\end{question}

Note that if $\CM$ is weakly o-minimal then failure of exchange in $\CM$ yields functions into valuational sorts which are not locally constant. Indeed, if $\alpha:M \to M $ is a definable locally constant function with infinite image then $y\mapsto \sup \,\alpha^{-1}(y):\Im(\alpha)\to M$ is injective and, outside finitely many points, takes valuational cuts as images.

\section{Generic Differentiability}

In \cite[Problem 3]{MacMaSt} it is asked whether definable functions in weakly o-minimal expansions of fields are generically differentiable. In \cite{Joh:exchange}, the fourth author gave an example of a weakly o-minimal expansion of a field in which generic differentiability fails. His example yields a definable function whose derivative is not in $K$ but in $\widebar K_{nv}$. In the present section we show that this example is the worst that could happen. Namely, that every definable function is generically differentiable in $\widebar K_{nv}$.

Fix, for the rest of this section, an $\aleph_0$-saturated weakly o-minimal $\CK$ expanding a real closed field $K$. Note that  the lateral derivatives of definable functions  exist in $\widebar K$. 

\subsection{Local Convexity}
A definable function $f$ on an interval $I$ is \emph{convex} if $m_f(x,y)<m_f(y,z)$ for all $x<y<z$ in $I$. 
   In the present section, we strengthen the monotonicity theorem for weakly o-minimal fields by showing that definable unary functions are, at generic points, either convex, concave or affine.

\begin{proposition}\label{Johnson convexity}
    For $I\sub K$ an open interval, let $f:I\to K$ be a definable function which is nowhere locally affine. Then $f$ is locally convex (or concave) at any generic point. 
\end{proposition}
\begin{proof} 
    Absorbing parameters into the language, we may assume that $f$ is $\0$-definable. 
    Let $a$ be a generic element in $I$. By our assumption on $f$ and weak o-minimality, for every $x\in I$, the intersection of the affine line connecting the points $(a,f(a))$ and $(x,f(x))$ with the graph of $f$ is finite. Thus, there is an  $\epsilon>0$, such that for all $x\in (a,a+\epsilon)$, the affine line through the points $(a,f(a))$ and $(x,f(x))$ does not intersect the graph of $f$ in the interval $(a,a+\epsilon)$.  By INP, we may assume that $\dim(a/\epsilon)=1$. It follows that there exists an interval $I_0\ni a$, such that for all $x,y\in I_0$, the affine line connecting the points $(x,f(x))$ and $(y,f(y))$ intersects the graph of $f$ only at $x$ and $y$. In particular, $m(a,-)$ is injective on $I_0$.

    Consider, in an elementary extension,  the convex sets  $J_0(a) =m(a,a^-)$ and $J_1(a)=m(a,a^+)$ (the images under $m(a,x)$ of $a^+$ and $a^-$). By our assumption on $I_0$, these two sets are disjoint. Thus, either $J_0(a)<J_1(a)$ or $J_0(a)>J_1(a)$. For concreteness, assume the former. 
    This implies that there is an interval $I_1\ni a$ such that for all $x,z\in I_1$, if $x<a<z$ then
    $m(x,a)<m(a,z)$. By INP, we may assume that $a$ is generic over the parameters defining $I_1$. Hence, $a$ is generic in \[\{y\in I_1: \forall x,z\in I_1(x<y<z\to m(x,y)<m(y,z)\}.\]

    By weak o-minimality, there is some open interval $I_2$, $a\in I_2\subseteq I_1$, such that for all $x<y<z$ in $I_2$, $m(x,y)<m(y,z)$, so $f$ is convex there.
\end{proof}

\subsection{Generic nv-differentiability}

\begin{definition}
    Assume that  $f: K\to K$ is a definable function and $a\in \dom(f)$. Then $f$ is \emph{nv-differentiable} at $a$ if \[\lim_{h\to 0}\frac{f(h+a)-f(a)}{h}\] exists in $\widebar K_{nv}$. 
\end{definition}

We will show that every definable unary function in $\CK$ is generically nv-differentiable. To do this we will first show that a definable function whose lateral derivatives are locally constant (as functions into  $\widebar K$) is generically locally affine. By Theorem \ref{T:func into val is lc}, this forces the lateral derivatives to be nonvaluational, and nv-differentiability then follows by standard techniques.

We define the lateral derivatives to be 
$f'_+(a)=\lim_{x\to a^+} \!m(a,x)$ and $f_-'(a)=\lim_{x\to a^-}\!m(a,x)$, both in $\widebar K$.

\begin{lemma}\label{L:f' constant is locally affine}
    Let $I$ be an open interval and $f:I\to K$ a definable function. 

    \begin{enumerate}
        \item If $f'_+(x)\equiv c\in \widebar K$  on $I$ then $c\in K$ and away from finitely many points, $f$ is locally of the form $f(x)=cx+b$, for $b\in K$. The same is true for $f_-'(x)$.
        \item there are at most finitely many points such that $f_+'(x)= \pm \infty$.
    \end{enumerate}
    \end{lemma}
\begin{proof}
Assume that $f_+'(x)\equiv c\in \widebar K$  on $I$ (the proof for $f_-'(x)$ is the same).
It is sufficient to show that there is a subinterval $J\subseteq I$ on which $f$ is $K$-affine with slope $c$.

If $f$ is locally constant, then $c=0$ and there is nothing to show; so assume that $f$ is locally strictly monotone. By Proposition \ref{Johnson convexity}, we may assume that $f$ is convex: For all $x<y<z$: $m(x,y)<m(y,z)$. It follows that for all $x<y<z$: 
\[
    m(x,y)< m(x,z)<m(y,z).
\]

This implies that $f'_+(x)\le m(x,z)$ and  $m(x,z)\leq f'_+(y)$. We conclude that for all $x,z$ we have $c\le m(x,z)\le c$, i.e., $m(x,z)$ is constant, for $x<z$ in $I$. But $m(x,z)\in K$, so $c\in K$, and $f$ is $K$-affine, with slope $c$.

(1) follows from the above. As for (2), if, for example, $f_+'(x)=\infty$ on infinitely many points then by weak o-minimality it is constant and equal to $\infty$ on an interval, contradicting (1).
\end{proof}

We say that a definable function $f: K\to K$ is \emph{piecewise $\mathcal P$} for some property $\mathcal P$ of (definable) functions, if there is a definable partition $I_1, \dots, I_k$ of $K$, such that for $1\le j\le k$ the restriction $f|I_j$ is $\mathcal P$. Before our main theorem we can conclude in those terms:: 

\begin{corollary}\label{C:piecewise}
    Let $f: K\to K$ be a definable function. If $f'(x)$ is locally constant then $f$ is piecewise $K$-affine. 
\end{corollary}
\begin{proof} By Lemma \ref{L:f' constant is locally affine}, after removing finitely many points from $\dom(f)$, we may assume that 
$f$ is locally $K$-affine. Namely, there are locally constant functions $b:K\to K$ and $c:K\to K$ such that for every $x\in \dom(f)$, $f(x)=c(x)\cdot  x+b(x)$. By exchange, the functions $b$ and $c$ have finite image, so the result follows.
\end{proof}

We are now ready to prove our main theorem:
\begin{theorem}\label{T:n=1 gen-dif}
    Let $f: K\to K$ be a definable partial function. Then $f$ is generically nv-differentiable. 
\end{theorem}
\begin{proof}
    Assume everything is definable over $\emptyset$.
    Fix some generic $a\in \dom(f)$. We have to show that $f$ is nv-differentiable at $a$. If $f$ is locally constant at $a$, there is nothing to prove, so we assume that is not the case; i.e. that it is strictly monotone around $a$. 

    The limit $f_+'(a)$ (and $f_-'(a)$) lies in $\widebar K$; it cannot be $\pm\infty$ by Lemma \ref{L:f' constant is locally affine}(2) and genericity of $a$.
    
    Next, we note that  $f'_+(a), f'_-(a)\in \widebar K_{nv}$. Indeed, otherwise, the function $x\mapsto f'_+(x)$ maps some interval around $a$ to a valuational sort. By Theorem \ref{T:func into val is lc} it is a locally constant function, and by Lemma \ref{L:f' constant is locally affine} $f$ is locally affine at $a$, implying that $f'(a)\in K$, a contradiction. 
    
    So both lateral derivatives are nonvaluational (possibly rational). If $f'_+(a)=f'_-(a)$, we are done. Towards a contradiction and without loss of generality, assume that $f'_-(a)<f'_+(a)$. Note that for no $f'_-(a)<d<f'_+(a)$ can we get $\dim(a/d)=1$. Indeed, if such a $d$ existed we would get that for the function $g(x)=f(x)-dx$, $g_+'(a)>0$ and $g_-'(a)<0$.  But, since $a$ is generic over the parameters defining $g$, the function $g$ must be continuous and either constant, strictly increasing or strictly decreasing around $a$. Neither option is compatible with both  $g_+'(a)>0$ and $g_-'(a)<0$, so we have reached a contradiction. But then  for all $f'_-(a)<d<f'_+(a)$ we have $a\in \dcl(d)$, which, since $\CK$ satisfies exchange, is absurd. 
\end{proof}

It is natural to ask whether, analogously, in weakly o-minimal fields the local monotonicity theorem can be strengthened as follows:
\begin{question}
    Let $K$ be a weakly o-minimal field, $f: K\to K$ a definable function. Is $f$ piecewise monotone? Note that a positive answer (applied to $f'$) would imply, moreover,  that $f$ is piecewise convex or concave. 
\end{question}

We can strengthen Theorem \ref{T:n=1 gen-dif} showing that, in fact, definable unary functions are \textit{strictly} nv-differentiable. 
For this we need the  following  standard lemma: 
\begin{lemma}
    Let $f:K\times K\to \widebar K $ be such that $f(a,h)$ is nonvaluational for all $(a,h)$. Assume that $a\in K$ is generic and $f(a,0)=\lim\limits_{h\to 0} f(a,h)$. Then $f$ is continuous at $(a,0)$, as a function of two variables. 
\end{lemma}
\begin{proof}
    Assume $f$ is definable over $\emptyset$. Fix $a\in K$ generic.  Note that, since $f(x,h)$ is nonvaluational the expression $|f(a,h)-f(a,y)|$ makes sense
    (and convergence of $f(a,h)$ to $f(a,0)$ can be interpreted using the standard $\epsilon, \delta$-definition). 
    
    Let $e>0$.  By INP we may assume that $\dim(a/e)=1$ and since $f(x,0)=\lim\limits_{h\to 0} f(x,h)$, there is $d>0$ such that $|f(a,h)-f(a,y)|<e$ for all $|y|<d$.  Again using INP,  we may assume that $\dim(a/d,e)=1$. So for all $x$ in a small enough neighborhood of $a$ we get $|f(x,h)-f(x,y)|<e$ for all $|y|<d$. Thus, $f$ is continuous at $(a,0)$. 
\end{proof}

\begin{lemma}\label{L: strict differentiability}
    If $f: K\to K$ is definable then it is generically strictly nv-differentiable, i.e., for generic $a\in K$ the limit \[
    \lim_{(x,y)\to (a,a)}\frac{f(x)-f(y)}{x-y}
    \] exists in $\widebar K_{nv}$. 
\end{lemma}
\begin{proof}
    By Theorem \ref{T:n=1 gen-dif} $f$ is generically nv-differentiable; hence we may assume it is nv-differentiable. Apply the previous lemma to $F(x,h):=\begin{cases}
        m(x,x+h) & h\neq 0 \\
        f'(x) & h=0
    \end{cases}$. 
\end{proof}

\subsection{Multivariable nv-differentiability}

The results of the previous subsection imply that, at generic points, definable multi-variable functions have (continuous) partial nv-derivatives. But in order to show that this implies multivariable nv-differentiability, some technicalities have to be addressed. We will be working, as above, in a slightly saturated weakly o-minimal expansion of a field, $\CK$.  Recall that $\widebar K_{nv}$ is a real closed field, so the following makes sense: 

\begin{definition}
   A  function $f:K^n\to (\widebar K_{nv})^k$  is  called \textit{$\widebar K_{nv}$-linear} if it is the restriction of a $\widebar K_{nv}$-linear map   $L:(\widebar K_{nv})^n\to (\widebar K_{nv})^k$. (Note that such a restriction is always definable, since scalar multiplication of elements in $K$ by elements of $\widebar K_{nv}$ is definable).
\end{definition}


Thus, we naturally have: 

\begin{definition}
    A definable function $f: K^n\to K^k$ is\textit{ nv-differentiable at a point $x_0$} if there exists a $\widebar K_{nv}$-linear map $L:K^n\to  (\widebar K_{nv})^k$  such that \[\lim\limits_{h\to 0}\frac{||f(x_0+h)-f(x_0)-L(h)||}{||h||}=0.\]
\end{definition}

 Note that if $f:K^n\to K$ is differentiable at $x_0$ then
the $\widebar K_{nv}$-linear map $L$ above consists of the partial derivatives of $f$ at $x_0$. We denote it by $D_{x_0}f=(\frac{\partial f}{x_1}(x_0),\ldots, \frac{\partial f}{x_n}(x_0))$.


\begin{theorem} \label{T:gen-dif}
Let $V\sub K^n$ be an open set and $f:V\to K$ an $A$-definable function. Then $f$ is  nv-differentiable at every generic point of $V$. 
\end{theorem}
\begin{proof} The proof is by induction on $n$, where the case $n=1$
is Theorem \ref{T:n=1 gen-dif}. Inductively, it is sufficient to prove:

\begin{claim} Assume that $I\times U \sub K\times K^{n-1}$ is open, $f:I\times U\to K$  a definable function. For  every $a\in I$, let $f_a(y)=f(a,y)$, and for every $\bar b\in U$, let $f^{\bar b}(x)=f(x,\bar b)$, and assume that each $f_a$ is nv-differentiable on $V$ and each $f^{\bar b}$ is nv-differentiable on $I$. Then, if $x_0\in I\times U$ is generic then  $f$ is nv-differentiable at $x_0$.  More precisely, 

 \[\lim\limits_{h\to 0}\frac{||f(x_0+h)-f(x_0)-D_{x_0}f(h)||}{||h||}=0.\]
\end{claim}
\begin{claimproof}

We may assume that all data is definable over $\emptyset$ and we fix a generic $(a,\bar b)\in I\times U$.  Notice that  $D_{a,\bar b}f=(D_a f^{\bar b},D_{\bar b}f_a)$ and hence $D_{a,\bar b}f(r,\bar s)=D_a f^{\bar b}(r)+D_{\bar b}f_a(\bar s)$

Fix $\epsilon>0$. By INP, we may assume that $\dim(a,\bar b/\epsilon)=n+1$.

By the differentiability assumption of $f_a$ and $f^{\bar b}$,   there exists $\delta>0$ such that for all $r$ and $s=(s_1,\ldots, s_n)$, with $|r|,|\bar s|<\delta $, 

\begin{equation}\label{eq: diff1}\begin{split} \mbox{ If } |r|,|\bar s|<\delta \mbox{ then } |f(a+r,\bar b)-f(a,\bar b)-D_af^{\bar b}(r)|<|r|\epsilon \\ 
|f(a,\bar b+\bar s)-f(a,\bar b)-D_{\bar b}f_a(\bar s)|<|\bar s|\epsilon.\end{split} \end{equation}
By INP, we may assume that $\dim(a,\bar b/\epsilon, \delta)=n+1$. Hence, by shrinking $I$ and $U$, we may assume that all $(a_1,\bar b_1)\in I\times U$ satisfy the above. Furthermore, by the continuity of the map $x\mapsto D_{\bar b}f_x$ at $a$ we may also assume that for all $x_1,x_2\in I$, we have $||D_{\bar b}f_{x_1}-D_{\bar b}f_x||<\epsilon$
(here the norm is computed as that of a linear map).

Thus, if $r,|s|<\delta$ then 
$|f(a+r,\bar b+\bar s)-f(a,\bar b)-D_{a,\bar b}f (r,\bar s)|$=$$ |(f(a+r,\bar b+\bar s)-f(a+r,\bar b)-D_{\bar b} f_a (\bar s))+(f(a+r,\bar b)-f(a,\bar b)-D_af^b(r))|$$ $$\leq |f(a+r,\bar b+\bar s)-f(a+r,\bar b)-D_{\bar b} f_a (\bar s)|+|r|\epsilon$$ $$\leq |f(a+r,\bar b+\bar s)-f(a+r,\bar b)-D_{\bar b}f_{a+r}(\bar s)|+ |(D_{\bar b}f_{a+r}-D_{\bar b} f_a )(\bar s)|+|r|\epsilon  $$ $$\leq |\bar s|\epsilon+|\bar s|\epsilon+|r|\epsilon\leq (2|\bar s|+|r|)\epsilon\leq 2|(r,\bar s)|\epsilon.$$  The result follows.
\end{claimproof}
 This ends the proof of the theorem.   
\end{proof}


It is now standard to show that under the same assumptions as those of Theorem \ref{T:gen-dif} the function $f$ is strictly differentiable on an open set of small co-dimension (see e.g., the analogous proof of \cite[Theorem 5.15]{JohP-minimal}). 

\section{Definable fields in weakly o-minimal expansions of fields}

\begin{theorem}\label{T:def fields}
    Let $\CK$ be a weakly o-minimal expansion of a field $K$. Let $F$ be a definable field in $\CK$ whose underlying set is a subset of $ (\widebar K_{nv})^n$, for some $n$. Then $F$ is  definably isomorphic to  either $K$ or $K^{alg}=K(\sqrt{-1})$.

    In particular, this holds when $F\sub K^n$, for some $n$.
\end{theorem}
\begin{proof}
    Let $F\sub (\widebar K_{nv})^n$ be a definable field; in particular it is interpretable in $\CK$, is  of finite dp-rank and consists of nonvaluational cuts. By Lemma \ref{L: F in nv is lsi to K}, there exists an infinite definable subset $S\subseteq F$ and a definable injection $S\hookrightarrow K$; in the language of \cite{HaHaPeVF}: $F$ is locally strongly internal to $K$. Also, in the language of \cite{HaHaPeVF}, $K$ is an SW-uniformity (\cite[Definition 3.2]{HaHaPeVF}). The only difference, not allowing us to apply \cite[Theorem 4.21]{HaHaPeVF} directly is that $K$ does not satisfy (Gen-Dif) per-se, i.e. the derivatives take values in $\widebar K_{nv}$ and not $K$. Nevertheless, the proof structure still applies almost verbatim. We highlight the main differences.

    In the above terminology, for a definable field $F$ that is locally strongly internal to an SW-uniformity, the proof of \cite[Theorem 4.21]{HaHaPeVF} produces a type-definable additive subgroup $\nu$ of $(F,+)$ (the so called infinitesimal subgroup). The group $\nu$ is invariant under scalar multiplication in the following sense: For $\widehat \CK\succ \CK$ a sufficiently saturated extension and $a\in F(\CK)$, $a\cdot \nu(\widehat \CK)=\nu(\widehat \CK)$. Furthermore, $\nu(\widehat \CK)$ has a structure of a $\widebar K_{nv}$-differential group  (see \cite[Proposition 4.19]{HaHaPeVF}), 
    specifically, for $a\in F(\CK)$, the Jacobian matrix of the map $a\mapsto a\cdot x$, at $0$, is a matrix with coefficients in $\widebar K_{nv}$, i.e. an element of $M_n(\widebar K_{nv})$ for some $n$. The chain-rule (which is applicable here as well) implies that there is an embedding of $F$, as a subring, into $M_n(\widebar K_{nv})$. One may definably embed $K$ into $M_n(\widebar K_{nv})$ as well, as scalar matrices. Inside $M_n(\widebar K_{nv})$, $K\cap F\subseteq K$ is a definable subfield, but by dimension considerations, $K$ has no proper subfields hence $K\subseteq F$. Since $K$ is real closed and $F$ has finite dp-rank, $F$ must be either $K$ or $K^{alg}$. 
\end{proof}

\appendix
\section{The mean value theorem}
O-minimal structures satisfy Rolle's Theorem and consequently Lagrange's mean value theorem for definable functions. It is easy to see that in weakly o-minimal expansions of fields definable (nv-) differentiable functions may fail to satisfy Rolle's theorem. In the present section we show, however, that a local version of the mean value theorem remains true.  

Given a convex set $I\sub K$ and a convex function $f:I \to K$, and $d\in K$, we let $$L_d=(m_f)^{-1}(d)=\{(x,y)\in K^2: m_f(x,y)=d\}\sub K^2.$$ Notice that the convexity of $f$ implies that $\dim(L_d)\leq 1$. Below we write $m(x,y)$ for $m_f(x,y)$.

In real analysis Darboux's theorem shows that if $f$ is a differentiable function on an open interval $I$ then $f'$ satisfies the mean value theorem on $I$. In weakly o-minimal theories Darboux's theorem fails, so its conclusion has to be added as an assumption to the following lemma:   

\begin{lemma}\label{L:curves and Diag v2}
Assume that $I\sub K$ is a convex set and $f:I \to K$ a convex nv-differentiable function. Assume further that $f'(I)$ is dense in $\conv(f'(I))$. Then, for every $d\in m(I\times I)$ outside a finite set, there exists a unique $c\in \widebar K$ such that $(c,c)\in \cl_{\widebar K}(L_d)$. Furthermore, $c$ is nonvaluational, and for $(x,y)\in L_d$ close enough to $(c,c)$, we have $x<c<y$.
    
\end{lemma}
\begin{proof}
The convexity of $f$, together with the non affiness, implies that for all $a<b\in I$,   $f'(a)<m(a,b)<f'(b)$. Thus, $m(I\times I)$ is contained in $\conv(f'(I))$ and  is dense in it.

For $e>0$ let $D_e=\{m(x,y):x<y\in I,\, y-x<e\}$. As above, $D_e$ is contained in $\conv(f'(I))$ and dense  in it, and thus $\cl_K(D_e)$ consists of a single convex component which equals $\cl_{\widebar K}(f'(I))\cap K$. Also, by definition of derivative, $f'(I)\sub \cl_{\widebar K}(D_e)$. Thus, $$\cl_{\widebar K}(m(I\times I))=\cl_{\widebar K}(f'(I))=\cl_{\widebar K}(D_e),$$ for every $e>0$.

It follows that  there is a finite set $F_e$, such that $m(I\times I)=D_e\cup F_e$, and $F_{e'}\supseteq F_e$, for $e'<e$. By weak o-minimality, for $e$ small enough, $F:=F_e$ remains the same. 

Fix $d\in m(I\times I)\setminus F$, and then $d\in D_e$ for all $e>0$.
the convexity of $f$ implies that $L_d^+:=L_d\cap \{(x,y)\in K^2: y>x\}$ is the graph of a strictly decreasing function  (note that $L_d$ is symmetric with respect to the diagonal of $K$). 

For $e>0$, let $$\Delta_e=\{(x,y)\in K^2:|x-y|<e\},$$ a tube around the diagonal, and then $D_e=m(\Delta_e)$, so $L_d\cap \Delta_e\neq \emptyset$, for all $e>0$. It follows that the frontier in $\widebar K^2$ of $L_d$ belongs to the diagonal, $\Delta_{\widebar K}$,  of $\widebar K$. In fact, if $c=\inf_{\widebar K}(\pi_2(L_d))$ (the projection onto the $y$-coordinate of $L_d$), then $Fr_{\widebar K}(L_d)\cap \Delta_{\widebar K}=(c,c)$.

Since $L_d$ is the graph of a strictly decreasing function, it follows from weak o-minimality, that for sufficiently small $e$, $L_e\cap \Delta_e$ is contained in the ``third quadrant'' with respect to $(c,c)$, namely in the set $\{(x,y)\in K^2: x<c<y\}$. In particular, since we have such $x<c<y$, with $y-x<e$,   for all sufficiently small $e>0$, it follows that $c$ is nonvaluational.
\end{proof}

We aim to prove a version of the mean value theorem. To state it we, first, need : 
\begin{definition}
    A definable function $f:K\to K$ is differentiable at a nonvaluational cut  $c$ if $f$ can be extended continuously to $c$ and $\lim\limits_{x\to c} \frac{f(x)-f(c)}{x-c}$ exists in $\widebar K $. 
\end{definition}

\begin{remark}
    It can be shown that for any definable function $f: K\to K$ there is a partition of $K$ into finitely many convex sets $I_1,\dots, I_k$ such that for all $1\le i\le k$ $f|I_i$ can be extended continuously to a function on the convex hull of $I_i$ in $\widebar K_{nv}$ (and into $\widebar K_{nv}$). So the assumption that $f$ can be extended to $c$ is harmless. In the application below, it will be automatic. Also note that, for a fixed $c$, the function $m_f(x,c):K\to K_{nv}$ is definable and therefore the lateral limits, at least, exist as $x$ tends to $c$. 
\end{remark}


\begin{theorem}\label{T: MVT}
    Let $f: I\to K$ be a $\0$-definable, convex ${nv}$-differentiable function,  with $I\sub K$ convex. 

    Then, there is a $\0$-definable partition $I=\bigcup I_j$ into definable  convex sets and a finite set $D\sub K$,  such that for each $j$, for all $a<b\in I_j$, if $m(a,b)\notin D$ then there exists $c\in \widebar K_{nv}$, $a<c<b$, such that
    $$f'(c)=\frac{f(b)-f(a)}{b-a}.$$
   \end{theorem}

\begin{proof}
First, we claim that after partitioning $I$ we may assume that for every $c\in I_{nv}$, either $f$ is strictly increasing on $c^-$ and $c^+$ or $f$ is strictly decreasing on both $c^-$ and $c^+$. Indeed, since $f$ is convex, it is strictly monotone, locally, at generic points of $K$, so by omitting finitely many points we may assume that $f$ is locally monotone at every point of $I$. Define the equivalence relation on $I$: $xEy$ if $f$ is strictly monotone on some open interval containing $[x,y]$. The $E$-classes must be convex subsets of $K$, and  by Theorem \ref{T:no lc into nv v2}, there are only finitely many $E$-classes whose endpoints (in $\widebar K$) are nonvaluational. Hence, after removing finitely many points from $I_{nv}$, and partitioning $I$ accordingly,  we may assume, without loss of generality, that $f$ is strictly increasing or decreasing on both $c^-$ and $c^+$, for all $c\in I_{nv}$.

By Corollary \ref{C:partition of nv}, we may partition $f'(I)$ into finitely many sets such that each one is dense in its $\widebar K$-convex hull. Using a weak o-minimality, we can therefore partition $I$ into finitely many convex $I_j$, such that on each $I_j$, the image $f'(I_j)$ is dense in $\conv(f'(I_j))$, as in the assumption of Lemma \ref{L:curves and Diag v2}. Thus, we assume now that $I$ satisfies these assumptions.

By the same lemma, for all $d\in m(I\times I)$, outside a finite set. there exists a nonvaluational $c$ such that for arbitrarily close $x,y$, with $x<c<y$, we have $m(x,y)=d$. (Note that, geometrically, this means that there are secant lines on the graph of $f$, arbitrarily close to $c$, whose slope is $d$).

We  claim that away from finitely many cuts $c\in I_{nv}$, $f$ can be extended continuously to $c$, as a function into $\widebar K_{nv}$, and in addition $f'(c)=d$, in the above sense. First note that the lateral limits of $f$ at $c$ exist in $\widebar K$ (by weak o-minimality) and must be equal. Indeed, if  the limits are not equal then for $x<c<y$, with $|x-y|\to 0$  we have $m(x,y)\to \infty$, contradicting the assumption that $(c,c)\in \cl(L_d)$.
Similarly, if the limit were a valuational cut, then since $f$ is, say, strictly increasing on both $c^-$ and $c^+$, then $m(x,y)\to \infty$, for $x<c<y$, with $|x-y|\to 0$. There can be at most finitely many $c\in I_{nv}$ where this can happen, by weak o-minimality.

Thus, $f$ can be continuously extended to $c$ and $f(c)$ (we still call this extension $f$), is nonvaluational. Next we claim that $f'(c)=d$.

By continuity of $f$,  for every $x_1<c$, $m(x_1,c)=\lim\limits_{x_2\to c^-}m(x_1,x_2)$, and also for every $x_2<c$, $m(x_1,x_2)<d=m(x_2,y_2)$, for some $y_2>c$. Hence, $f'(c^-)\leq d$.
Similarly, $f'(c^+)\geq d$. 

By the convexity of $f$, for each $c_1,c_2\in K$, if $c_1<c<c_2$ then $f'(c_1)< f'(c^-)\leq f'(c^+)< f'(c_2)$. Since $f'(I)$ is dense in $\conv(f'(I))$ it follows that $f'(c^-)=f'(c^+)=d$.
\end{proof}


   Note that, in the notation of the theorem, it could be that $0\in D$, so that the theorem does not imply a weakly o-minimal version of Rolle's theorem. 

    \bibliographystyle{plain}
    \bibliography{harvard}
    
    \end{document}